\documentclass[12pt]{amsart}
\usepackage{amscd,amssymb}

\usepackage{hyperref}

\numberwithin{equation}{section}

\newcommand{\fullref}[1]{\ref{#1} on page~\pageref{#1}}

\newcommand{\ndash}{\nobreakdash-\hspace{0pt}}
\newcommand{\Ndash}{\nobreakdash--}

\newcommand{\dd}{{\mathrm{d}}}

\newcommand{\hagrid}{{\calC_\pi(M;C_0,C_1)}}
\newcommand{\hagr}[2]{{\calC_\pi(M;{#1},{#2})}}
\newcommand{\hagridbis}{{\calC_\pi(M;C_1,C_2)}}

\newcommand{\fluffy}{{\calC_\pi(M)}}

\newcommand{\rubeus}{{\calC_\pi^{M;C_0,C_1}}}
\newcommand{\rub}[2]{{\calC_\pi^{M;{#1},{#2}}}}

\newcommand{\Plie}{P_{C_0,C_1}\Omega^1(M)}

\newcommand{\Ozero}{\Omega^0(I,T_{X(0)}M)}
\newcommand{\Oone}{\Omega^1(I,T^*_{X(0)}M)}
\newcommand{\Ob}{\Omega^0(I,T^*_{X(0)}M)}

\DeclareMathOperator{\Ann}{Ann}
\DeclareMathOperator{\Ima}{Im}

\DeclareMathOperator{\EL}{EL}

\newcommand{\id}{\mathrm{id}}

\newcommand{\toto}{\rightrightarrows}
\newcommand{\up}{\underline{p}}

\DeclareMathOperator{\End}{End}
\DeclareMathOperator{\Hom}{Hom}
\DeclareMathOperator{\Iso}{Iso}

\newcommand{\tpm}{T^*PM(C_0,C_1)}

\newtheorem{Thm}{Theorem}[section]
\newtheorem{Prop}[Thm]{Proposition}
\newtheorem{Lem}[Thm]{Lemma}
\newtheorem{Cor}[Thm]{Corollary}

\newtheorem*{Thm*}{Theorem}
\newtheorem*{Lem*}{Lemma}

\theoremstyle{remark}
\newtheorem{Rem}[Thm]{Remark}
\newtheorem*{Ack}{Acknowledgment}

\theoremstyle{definition}

\newtheorem{Exa}[Thm]{Example}

\newcommand{\braket}[2]{\left\langle{\,{#1}\,,\,{#2}\,}\right\rangle}

\newcommand{\Lie}[2]{{\left[{\,{#1}\,,\,{#2}\,}\right]}}

\newcommand{\G}{\mathsf{G}}
\newcommand{\C}{\mathsf{C}}

\newcommand{\bbR}{{\mathbb{R}}}

\newcommand{\de}{\partial}

\newcommand{\calC}{\mathcal{C}}

\newcommand{\calM}{\mathcal{M}}

\newcommand{\calF}{\mathcal{F}}

\newcommand{\be}{{\mathbf{e}}}

\DeclareMathOperator{\LLL}{L}

\def\gpd{\,\lower1pt\hbox{$\longrightarrow$}\hskip-.24in\raise2pt
               \hbox{$\longrightarrow$}\,}

\let\Tilde=\widetilde

\let\Hat=\widehat

\DeclareMathOperator{\im}{im}

\newcommand\qq{}

\newcommand\lmp[1]{{\qq Lett.\ Math.\ Phys.\ \bf #1}}

\newcommand\anm[1]{{\qq Ann.\ Math.\ \bf #1}}

\newcommand\jdg[1]{{\qq J.\ Diff.\ Geom.\ \bf #1}}

\newcommand\travm[1]{{\qq Travaux ma\-th\'e\-ma\-ti\-ques \bf #1}}

\begin{document}
\title
{Coisotropic Submanifolds and Dual Pairs}


\author[A.~S.~Cattaneo]{Alberto~S.~Cattaneo}
\address{Institut f\"ur Mathematik, Universit\"at Z\"urich\\
Winterthurerstrasse 190, CH-8057 Z\"urich, Switzerland}  
\email{alberto.cattaneo@math.uzh.ch}




\thanks{A.~S.~C. acknowledges partial support of SNF Grant No.~200020-131813/1.}


\begin{abstract}
The Poisson sigma model is a widely studied two\ndash dimensional topological field theory.
This note shows that boundary conditions for the Poisson sigma model are
related to coisotropic submanifolds 
(a result announced in [math.QA/0309180])
and that
the corresponding reduced phase space is a (possibly singular) dual pair
between the reduced spaces of the given two coisotropic submanifolds.
In addition the generalization to a more general tensor field is considered and it is shown
that the theory produces Lagrangian evolution relations if and only if the tensor field is Poisson.
\end{abstract}

\maketitle

\tableofcontents

\section{Introduction}\label{intro}
Let $M$ be a finite-dimensional manifold and $I$ the unit interval $[0,1]$. We denote by 
$PM:=C^1(I,M)$ the space of differentiable paths on $M$ and by $T^*PM$ the space
of bundle maps $TI\to T^*M$ with continuously differentiable base map and continuous fiber map.
We consider $PM$ as a Banach manifold and $T^*PM$ as a Banach, weak symplectic manifold. (A weak symplectic form is a closed
$2$\ndash form that induces an injective map from the tangent to the cotangent bundle.)
The canonical symplectic form $\Omega$ is the differential of the canonical $1$\ndash form
$\Theta$. If we denote an element of $T^*PM$ by $(X,\eta)$ where $X$ is the differentiable base map and
the fiber map $\eta$ is regarded as a continuous section of $T^*I\otimes X^*T^*M$, we have
\[
\Theta(X,\eta)(\Hat\xi)=\int_I \braket\eta\xi,
\]
where $\Hat\xi$ is a tangent vector at $(X,\eta)$, $\xi$ is 
its projection to $T_XPM=\Gamma(X^*TM)$,
and $\braket{\ }{\ }$ is
the canonical pairing between the cotangent and the tangent bundles to $M$.

Let now $\pi$ be a section of $TM\otimes TM$. We denote by $\pi^\sharp$ the induced bundle map
$T^*M\to TM$ satisfying 
\[
\braket{\pi^\sharp(x)\sigma}\tau=\pi(x)(\sigma,\tau), \quad
\forall x\in M,\quad
\forall\sigma,\tau\in T^*_xM.
\]
Given two submanifolds $C_0$ and $C_1$ of $M$, we denote by $\hagrid$ the space of 
``$\pi$\ndash compatible paths'' from $C_0$ to $C_1$; i.e.,
\begin{multline*}
\hagrid:=\{(X,\eta)\in T^*PM :
\dd X + \pi^\sharp(X)\eta = 0,\\
X(0)\in C_0,\ X(1)\in C_1\},
\end{multline*}
where the differential $\dd X$ of the base map is regarded as a section of 
$T^*I\otimes X^*TM$.
By using the implicit function theorem, one can easily prove \cite{CF01} 
that $\hagr MM$ is a Banach submanifold
of $T^*PM$. 
In general, for other submanifolds $C_0$ and $C_1$, 
$\hagrid$ is not a Banach submanifold.\footnote{The simplest example is when $\pi$ is identically zero.
In this case $\hagrid$ is a fibration over $B:=C_0\cap C_1$ with fiber at $x$ given by 
$\Omega^1(I,T^*_xM)$. If the basis $B$ is not a manifold, so neither is $\hagrid$.}

Anyway, even if $\hagrid$ is not a submanifold, we may define its Zarisky tangent space at each 
point. Namely, we first define the submanifold $\tpm$ of $T^*PM$ consisting of
bundle maps whose base maps connect $C_0$ to $C_1$. Then the tangent space at $(X,\eta)\in\hagrid$
is defined as the subspace of vectors in $T_{(X,\eta)}\tpm$ satisfying the linearized equation.

In general, we will call subvariety (of a smooth manifold) the common 
zero set of a family of smooth functions
and, 
by abuse of notation, we will call tangent bundle the union of the Zarisky tangent spaces to a 
subvariety. 
Given a symplectic manifold $(\calM,\omega)$ and a subvariety $\calC$,
we define, again by abuse of notation,
the symplectic orthogonal bundle $T^\perp\calC$ to $\calC$ as the union
$T_x^\perp\calC$, $x\in\calC$, with
\[
T_x^\perp\calC:=\{v\in T_x\calM : \omega_x(v,w)=0\ \forall w\in T_x\calC\}.
\]
The subvariety $\calC$ is said to be coisotropic if $T^\perp\calC\subset T\calC$.

\subsection{The main results}
In Sect.~\ref{s:tangent}
we study the tangent bundle and the symplectic orthogonal
bundle to $\hagrid$. 
Relying on the results of Sect.~\ref{s:tangent}, we prove in subsection~\ref{s:proof1}
(viz., Props.~\ref{oneone} and~\ref{onetwo}) the following
\begin{Thm}\label{t:one}
$\hagr MM$ is coisotropic 
in $T^*PM$ if{f} $\pi$\/ is a Poisson bivector field.
\end{Thm}
Recall that a Poisson bivector field $\pi$ is a skew-symmetric $2$\ndash tensor field
satisfying $\Lie\pi\pi=0$, where $\Lie{\ }{\ }$ is the Schouten--Nijenhuis bracket.
A proof of
the if-part of the Theorem is contained in \cite{CF01}, but the result was already known
\cite{SS,I} in the case when one considers loops instead of paths.

If $\calC$ is a coisotropic submanifold of a weak symplectic Banach manifold,
$T^\perp\calC$ is an integrable distribution 
and the leaf space $\underline\calC$, 
also called the reduced phase space, inherits a weak symplectic structure if it is a smooth manifold.
The reduced phase space $\underline{\hagr MM}$ has been shown in \cite{CF01} to also have the structure
of a topological groupoid; if it is smooth, 
it is a symplectic groupoid integrating the Poisson manifold
$M$. 


Now notice that not all possible
boundary conditions allow solutions to the constraint equation. In other words,
the maps
\begin{equation}\label{e:p}
p_i\colon\hagrid\to C_i,
\qquad i=0,1,
\end{equation}
associating $X(i)$ to a solution $(X,\eta)$ may not
be surjective. With this notation we may give the precise formulation, and a proof 
(see subsection~\ref{s:proof2}, in particular Props.~\ref{twoone} and~\ref{twotwo}), of a result
announced in \cite{CF03} as Theorem~3.1 (implicitly assuming there that the $C_i$s 
were chosen so that the $p_i$s were surjective):
\begin{Thm}\label{t:two}
If $\pi$ is a Poisson bivector field,
then $\hagrid$  is coisotropic 
in $T^*PM$ if{f}\/ $\Ima p_0$ and\/ $\Ima p_1$ are coisotropic in $M$ relative to $C_0$ and $C_1$
respectively.
\end{Thm}
Recall that a submanifold $C$ of a finite dimensional manifold $M$ endowed with a $2$\ndash tensor field $\pi$
is called coisotropic if $\pi^\sharp(N^*C)\subset TC$, where $N^*C$ denotes the conormal bundle to $C$:
\[
N^*_xC:=\{\alpha\in T^*_xM : \braket\alpha v=0\ \forall v\in T_xC\},\quad \forall x\in C.
\]
If $S$ is a subset of a submanifold $C$ of a Poisson manifold $M$, 
we say that $S$ is coisotropic in $M$ relative to $C$ if $\pi^\sharp(N^*_xC)\subset T_xC$
for all $x\in S$. We consider the empty set as a coisotropic submanifold of any symplectic or Poisson
manifold.

Observe that having $\hagrid$ coisotropic may not be enough to conclude that $\pi$ is Poisson
as it is possible that the base paths contained in $\hagrid$ do not explore the whole of 
$M$.\footnote{For example, assume that $\pi$ has a zero at a point $x$. Then $\hagr {\{x\}}M$
consists of pairs $(X,\eta)$ where $X$ is the constant path at $x$ and there are no conditions on $\eta$.
It is then clear that $\hagr {\{x\}}M$ is coisotropic (actually Lagrangian) whatever the tensor $\pi$ is.}

By definition we may also
reformulate the if-part of Theorem~\ref{t:two} in the following (slightly weaker) form:
\begin{Cor}\label{cor:cor}
If  $\pi$ is Poisson and 
$C_0$ and $C_1$ are coisotropic,
then $\hagrid$ is coisotropic. 
\end{Cor}
In this case, as recalled above, $T^\perp\hagrid$ is an integrable distribution on $\hagrid$,
which we describe in details in subsubsection~\ref{s:sym}.

In Section~\ref{s:dual} we concentrate on the case when $C_0$ and $C_1$ are coisotropic submanifolds of a Poisson
manifold $M$ and discuss how the reduction of $\hagrid$ may be understood as a (singular) dual pair.

Finally in Section~\ref{s:fd} we show that  this reduction  
may also be recovered
by Poisson reduction of the coisotropic submanifold of the symplectic groupoid of $M$ determined by
the intersection of the preimages of $C_0$ and $C_1$ under the source and target maps, respectively.




\begin{Rem}
One may consider the weaker condition that $\hagrid$ be presymplectic (i.e., that the kernel of the restriction of the symplectic form is a subbundle of the tangent
bundle of $\hagrid$). As shown in \cite{CalFal}, a sufficient condition for this to happen is that $C_0$ and $C_1$ are pre-Poisson submanifolds of $M$
(according to the definition in \cite{CZ}).
We will not elaborate on this in this paper.
\end{Rem}

\subsection{The case of the circle}\label{ss:caci}
In this paper we mainly work on the path space $PM$ as it is interesting to have boundary components (to be associated to the submanifolds $C_0$ and $C_1$).
A very similar, actually a bit easier, story works in the case of the loop space $LM:=C^1(S^1,M)$. In this case, we can analogously define the weak symplectic Banach manifold $T^*LM$ as the space
of bundle maps $TS^1\to T^*M$ with continuously differentiable base map and continuous fiber map. The weak symplectic form $\Omega$ is again the differential of the canonical $1$\ndash form $\Theta$:
\[
\Theta(X,\eta)(\Hat\xi)=\int_{S^1} \braket\eta\xi,
\]
where we use the same notations as above. To the tensor $\pi$ we now associate the Banach submanifold
\[
\calC_\pi(M):=
\{(X,\eta)\in T^*LM :
\dd X + \pi^\sharp(X)\eta = 0
\}.
\]
A very similar proof to Theorem~\ref{t:one}, see subsection~\ref{s:proof3}, yields the following
\begin{Thm}\label{t:onecircle}
$\calC_\pi(M)$ is coisotropic 
in $T^*LM$ if{f}\/ $\pi$ is a Poisson bivector field.
\end{Thm}
In this case, as recalled above, $T^\perp\calC_\pi(M)$ is an integrable distribution on $\calC_\pi(M)$,
which we describe in details in subsubsection~\ref{s:symc}.

\subsection{Lagrangian field theories with boundary}\label{ss:LFT}
In \cite{CMR} and \cite{CMR2} the general notion of Lagrangian field theories on manifolds with boundary is studied. The symplectic manifolds $T^*PM$ and $T^*LM$ described above arise as spaces of boundary fields of a two\ndash dimensional Lagrangian field theory with $\hagrid$ and $\fluffy$ as its ``spaces of Cauchy data."
The requirement of $\pi$ being Poisson turns out to be equivalent to the requirement that the theory is ``good" in the sense that the the evolution relations determined by solutions to the Euler--Lagrange equations are (immersed) Lagrangian submanifolds. We discuss this in more details in Section~\ref{s:LFT}.

\begin{Ack}
I thank I. Contreras for useful discussions and comments.
\end{Ack}

\section{Tangent and orthogonal bundles}\label{s:tangent}
By choosing a linear connection on $M$, which for simplicity we assume to be torsion free, 
we may identify the tangent bundle to $T^*M$ with the vector bundle
$E:=T^*M\oplus TM\oplus T^*M$. Explicitly this is done as follows. 
First observe that both $TT^*M$ and $E$
can be regarded as vector bundles  over $T^*M$ with fiber at point $(x,p)$
given by the vector space $T_xM\oplus T^*_xM$; the transition functions in the two vector
bundles are however different.
Choosing on $M$ local coordinates $\{x^i\}_{i=1,\dots,m}$, $m=\dim M$, 
and the corresponding dual coordinates $\{p_i\}$ on $T^*_xM$, 
we consider the fiber isomorphism 
\[
\Phi_{(x,p)}\colon
\begin{array}[t]{ccc}
T_{(x,p)}T^*M=T_xM\oplus T^*_xM &\to & E_{(x,p)}=T_xM\oplus T^*_xM\\
(\dot x^i,\dot p_i) &\mapsto & (\dot x^i, \dot p_i - \Gamma^r_{si}(x)\,p_r\,\dot x^s)
\end{array}
\]
where the $\Gamma$s are the Christoffel symbols of the given connection,
and we use Einstein's convention that a sum over upper and lower repeated indices is understood.
Then $\Phi\colon TT^*M\to E$ is a vector bundle isomorphism.

By this we may also identify $TT^*PM$ with $T^*PM\oplus TPM\oplus T^*PM$, 
regarded as vector bundles over $T^*PM$. Recall that
$T_XPM=\Gamma(X^*TM)$ and $T^*_XPM:=\Gamma(T^*I\otimes X^*T^*M)$.\footnote{To avoid cumbersome notations, 
from now on we will avoid indicating which maps or forms are continuous or differentiable.}
To describe this isomorphism 
explicitly, we observe that, given a continuous path $X$, we may subdivide the interval
$I$ into finitely many subintervals $I_\alpha$ such that $X(I_\alpha)$ is contained in a coordinate
patch $\forall\alpha$.  For a given $I_\alpha$,
we denote by $X^i$ and $\eta_i$, $i=1,\dots,m:=\dim M$, the components in local coordinates
of the restrictions to $I_\alpha$ of $X$ and $\eta$. The restriction to $I_\alpha$
of a tangent vector $\Hat\xi$ can then be split correspondingly 
into its components $\xi^i$ in the $X^i$\ndash direction
and $\zeta_i$ in the $\eta_i$\ndash direction. 
We then define
\begin{equation}\label{e:e}
e_i = \zeta_i - \Gamma^r_{si}(X)\,\eta_r\,\xi^s.
\end{equation}
The map $(X^i,\eta_i,\xi^i,\zeta_i)\mapsto(X^i,\eta_i,\xi^i,e_i)$ is well-defined globally
and yields the required vector bundle isomorphism $TT^*PM\to T^*PM\oplus TPM\oplus T^*PM$.
If we now impose boundary conditions, by the same map we may finally identify the fiber at $(X,\eta)\in\tpm$ with
\[
\{\xi\oplus e\in \Gamma(X^*TM) \oplus \Gamma(T^*I\otimes X^*T^*M) :
\xi(0)\in T_0,\ \xi(1)\in T_1\},
\]
with 
\[
T_0:=T_{X(0)}C_0\text{\/ and }T_1:=T_{X(1)}C_1.
\]

The linear connection on $M$ also induces a connection on the vector bundle $X^*TM$.
We denote by $\partial\colon \Gamma(X^*TM)\to\Gamma(T^*I\otimes X^*TM)$ 
the corresponding covariant exterior derivative. In local coordinates, its action
on a section $\sigma$ of $X^*TM$ is given by
\begin{equation}\label{e:DD}
(\partial\sigma)^i=\dd\sigma^i+\Gamma^i_{rs}\,\dd X^r\,\sigma^s.
\end{equation}
It is convenient to
modify this connection by using
\begin{equation}\label{e:AA}
A:=\nabla\pi^\sharp(X)\eta\in\Gamma(T^*I\otimes X^*\End(TM)),
\end{equation}
where $\nabla$ denotes the covariant derivative. In local coordinates we have
\begin{equation}\label{e:nablapi}
(\nabla\pi)^{ij}_k = \de_k\pi^{ij}
+\Gamma^i_{kr}\pi^{rj}+\Gamma^j_{kr}\pi^{ir}
\end{equation}
and
\begin{equation}\label{e:A}
A^i_k=\eta_j(\nabla\pi)_k^{ji}(X).
\end{equation}
We will denote by $D$ the covariant exterior derivative $\partial+A$.

\subsection{The tangent spaces to compatible paths}\label{ss:ts}
Let $(X,\eta)$ be a point in $\hagrid$. 
\begin{Prop}\label{p:conn} 
After choosing a connection on $M$ and using the above notations, we have
\begin{multline*} 
T_{(X,\eta)}\hagrid = \{\xi\in\Gamma(X^*TM),\ e\in\Gamma(T^*I\otimes X^*T^*M) : \\
(D\xi)  +  \pi^\sharp(X)e =0,\ 
\xi(0)\in T_0,\ \xi(1)\in T_1\}.
\end{multline*}
\end{Prop} 
\begin{proof}
Let us restrict our attention to a subinterval $I_\alpha$ such that $X(I_\alpha)$ is contained
in a coordinate patch.
The restriction to $I_\alpha$ of the equation
satisfied by $X$ and $\eta$ then reads
\begin{equation}\label{e:dX}
\dd X^i + \eta_j\,\pi^{ji}(X)=0
\end{equation}
in local coordinates. Let $\xi^i$ and $\zeta_i$ denote the local-coordinate expression of
a tangent vector. They then satisfy the equation
\[
\dd \xi^i +\eta_j\,\de_l\pi^{ji}(X)\,\xi^l  +  \zeta_j\,\pi^{ji}(X) =0,
\]
or equivalently by \eqref{e:nablapi}
\[
(\dd \xi^i -\xi^l\,\Gamma^i_{lr}(X)\,\eta_j\,\pi^{jr})
+\eta_j\,(\nabla\pi)_l^{ji}(X)\,\xi^l  +  
(\zeta_j-\Gamma^s_{lj}(X)\,\eta_s\,\xi^l)\,\pi^{ji}(X) =0.
\]
Since the connection is torsion free,
by \eqref{e:dX}, \eqref{e:e}, \eqref{e:DD}, \eqref{e:AA} and \eqref{e:A}
we conclude the proof.
\end{proof}

We then consider the parallel transport $U\in\Gamma(\Iso(X^*TM,T_{X(0)}M))$ 
of the connection $D$, viz., the solution
to the Cauchy problem
\begin{equation}\label{e:UU}
\left\{
\begin{aligned}
\dd\circ U &= U\circ D,\\
U(0)&= \mathit{Id}. 
\end{aligned}
\right.
\end{equation}
In local coordinates, we may also write
\begin{equation}\label{e:U}
\left\{
\begin{aligned}
\dd U^i_j &= U^i_l (A^l_j+\Gamma^l_{sj}\,\dd X^s),\\
U(0)^i_j &= \delta^i_j. 
\end{aligned}
\right.
\end{equation}
We may then simplify the equation
satisfied by the tangent vector $(\xi,e)$ into
\begin{equation}\label{e:dlambda}
\dd\lambda + P^\sharp\phi=0,
\end{equation}
with
\begin{subequations}\label{e:U*}
\begin{align}
\lambda &:= U\xi \in\Ozero,\label{Ua}\\
\phi &:= (U^t)^{-1}e\in\Oone,\label{Ub}\\
P^\sharp &:= U\pi^\sharp U^t\in\Omega^0(I,\Hom(T_{X(0)}^*M,T_{X(0)}M)).\label{Uc}
\end{align}
\end{subequations}
So we get
\begin{multline}\label{e:Ttwisted}
T_{(X,\eta)}\hagrid \cong
T_{(X,\eta)}\hagrid^\text{twisted} :=\\
\{\lambda\in\Ozero,\ \phi\in \Oone : 
\dd\lambda +P^\sharp\phi=0,\\ 
\lambda(0)\in T_0,\ U(1)^{-1}\lambda(1)\in T_1\}.
\end{multline}
Equation \eqref{e:dlambda} may be easily solved for any $\phi$ just assigning the initial
condition $\lambda(0)=\lambda_0$:
\begin{equation}\label{e:lambda}
\lambda(u)=\lambda_0-\int_0^u P^\sharp\phi.
\end{equation}
Se we get the alternative description
\begin{multline*}
T_{(X,\eta)}\hagrid \cong
T_{(X,\eta)}\hagrid_0 :=\\
\left\{\lambda_0\in T_0,\ \phi\in \Oone : 
U(1)^{-1}\left(\lambda_0-\int_I P^\sharp\phi\right)\in T_1\right\}.
\end{multline*}

\subsubsection{Properties of $P$}\label{p:P}
The tensor $\pi$ has been replaced by $P$ in \eqref{Uc}. Just by differentianting, 
it is not difficult to see that $P^\sharp$ is the solution to the Cauchy problem
\[
\left\{
\begin{aligned}
\dd P^\sharp &= U T^\sharp U^t,\\
P^\sharp(0) &= \pi_0^\sharp, 
\end{aligned}
\right.
\]
with $\pi_0:=\pi(X(0))$ and $T^\sharp:=D\pi^\sharp$.
Using \eqref{e:dX}, \eqref{e:A} and \eqref{e:U}, we obtain
in local coordinates
\[
T^{ls}=\eta_k\,(\pi^{rs}(X)\,(\nabla\pi)_r^{kl}(X)
-\pi^{kr}(X)\,(\nabla\pi)_r^{ls}(X)
+\pi^{lr}(X)\,(\nabla\pi)_r^{ks}(X)).
\]
Recall that, in local coordinates, 
the vanishing of the Schouten--Nijenhuis bracket of a bivector field $\pi$ may also be written,
by using any connection, as
\[
\pi^{sr}\,(\nabla\pi)_r^{lk}
+\pi^{kr}\,(\nabla\pi)_r^{sl}
+\pi^{lr}\,(\nabla\pi)_r^{ks}=0.
\]
This immediately implies the following
\begin{Lem}\label{l:piP}
If $\pi$ is a Poisson bivector field, then $P=\pi_0$.
\end{Lem}
Observe that $P$ depends on the chosen $(X,\eta)$. We also have the following
\begin{Lem}\label{l:Ppi}
If $P$ is skew-symmetric and constant for all $(X,\eta)\in\hagr MM$,
then $\pi$ is a Poisson bivector field.
\end{Lem}
\begin{proof}
If $P$ is skew symmetric, then so is $\pi_0$, that is, $\pi$ at any possible starting point of a path
$X$. Thus, $\pi$ is a bivector field. Moreover, for any $x\in M$, we can choose a solution to
\eqref{e:dX} with $X$ passing through $x$ for some $u_0\in I$
and, in a neighborhood of $u_0$,
$\eta_j=\be_j\dd u$, with $\be_j$ a basis element of $(\bbR^m)^*$. This implies that $T$ vanishes at $u_0$ and hence that
\[
\pi^{rs}(x)\,(\nabla\pi)_r^{kl}(x)
-\pi^{kr}(x)\,(\nabla\pi)_r^{ls}(x)
+\pi^{lr}(x)\,(\nabla\pi)_r^{ks}(x) = 0.
\]
Since this holds for all $x\in M$, it follows that $\pi$ is Poisson.
\end{proof}

\subsection{The symplectic orthogonal spaces to compatible paths}\label{s:sos}
Assuming that the chosen connection is torsion-free,
the symplectic form $\Omega$ evaluated at
tangent vectors $(\xi,\zeta)$ and $(\Tilde\xi,\Tilde\zeta)$
to $T^*PM$ at a point $(X,\eta)$ reads
\[
\Omega_{(X,\eta)}((\xi,\zeta),(\Tilde\xi,\Tilde \zeta))=
\int_I \braket e{\Tilde\xi} -\braket{\Tilde e}\xi,
\]
where $(\xi\oplus e)$ and $(\Tilde\xi\oplus\Tilde e)$ are the corresponding
elements of $T_XPM\oplus T^*_XPM$.
By the transformation \eqref{Ua} and \eqref{Ub}, and the analogous ones 
$\Tilde\lambda=U\Tilde\xi$, $\Tilde\phi=(U^t)^{-1}\Tilde e$, we get
\begin{equation}\label{e:ana}
\Omega_{(X,\eta)}((\xi,\zeta),(\Tilde\xi,\Tilde \zeta))=
\int_I \braket\phi{\Tilde\lambda} -\braket{\Tilde \phi}\lambda.
\end{equation}
Assume now that $(\xi,e)$ is tangent to $\hagrid$. Then, by \eqref{e:lambda},
\begin{multline*}
\int_I\braket{\Tilde \phi}\lambda=
\braket{\int_I\Tilde\phi}{\lambda_0}
-\int_I\braket{\Tilde\phi}{\int_0^\bullet P^\sharp\phi}=\\
=\braket{\int_I\Tilde\phi}{\lambda_0}
-\int_I \braket\phi{(P^\sharp)^t\int_\bullet^1\Tilde\phi}.
\end{multline*}
We thus obtain
\begin{multline*}
T^\perp_{(X,\eta)}\hagrid \cong 
T^\perp_{(X,\eta)}\hagrid^\text{implicit} :=\\
\left\{\Tilde\lambda\in\Ozero,\ 
\Tilde\phi\in \Oone : \phantom{\int}\right.\\
\int_I\braket\phi{\Tilde\lambda+(P^\sharp)^t\int_\bullet^1\Tilde\phi}
-\braket{\int_I\Tilde\phi}{\lambda_0}=0,\\
\left.\phantom{\int}
\forall(\lambda_0,\phi)\in T_{(X,\eta)}\hagrid_0\right\}.
\end{multline*}

\section{Proofs to the main theorems}
Using the results and notations of Sect.~\ref{s:tangent}, we are now going to prove the main Theorems~\ref{t:one}
and~\ref{t:two} and
to draw further consequences.
\subsection{Proof of Theorem~\ref{t:one}}\label{s:proof1}
In the case $C_0=C_1=M$, we have
\[
T_{(X,\eta)}\hagr MM_0 :=\\
\{\lambda_0\in T_{X(0)}M,\ \phi\in \Oone\}.
\]
Thus, $(\Tilde\lambda,\Tilde\phi)$ belongs to $T^\perp_{(X,\eta)}\hagr MM^\text{implicit}$
if{f}
\[
\int_I\braket\phi{\Tilde\lambda+(P^\sharp)^t\int_\bullet^1\Tilde\phi}
-\braket{\int_I\Tilde\phi}{\lambda_0}=0
\]
for all $\lambda_0\in T_{X(0)}M$ and $\phi\in\Oone$.
This implies that $(\Tilde\lambda,\Tilde\phi)$ belongs to $T^\perp_{(X,\eta)}\hagr MM^\text{implicit}$
if{f}
\begin{equation}\label{e:tlambda}
\Tilde\lambda(u)+(P^\sharp)^t\int_u^1\Tilde\phi=0
\end{equation}
and
\begin{equation}\label{e:tphi}
\int_I\Tilde\phi=0.
\end{equation}
Now we have
\begin{Prop}\label{oneone}
If $\pi$ is a Poisson bivector field, then  $\hagr MM$ is coisotropic.
\end{Prop}
\begin{proof}
By Lemma~\ref{l:piP} we have $P=\pi_0$. So
\eqref{e:tlambda} implies that $(\Tilde\lambda,\Tilde\phi)$ belongs to 
$T_{(X,\eta)}\hagr MM^\text{twisted}$. Thus, $\hagr MM$ is coisotropic.
\end{proof}
\begin{Prop}\label{onetwo}
If $\hagr MM$ is coisotropic, then  $\pi$ is a Poisson bivector field.
\end{Prop}
\begin{proof}
Since $T^\perp_{(X,\eta)}\hagr MM\subset T_{(X,\eta)}\hagr MM$, any pair 
$(\Tilde\lambda,\Tilde\phi)$ satisfying \eqref{e:tlambda} and \eqref{e:tphi} also belongs
to $T_{(X,\eta)}\hagr MM^\text{twisted}$; i.e., it satisfies
\[
\dd\Tilde\lambda+P^\sharp\Tilde\phi=0.
\]
On the other hand, differentiating \eqref{e:tlambda} yields
\[
\dd\Tilde\lambda+\dd (P^\sharp)^t\int_\bullet^1\Tilde\phi   
-(P^\sharp)^t\Tilde\phi   =0.
\]
So we get,
\begin{equation}\label{e:voldemort}
\dd (P^\sharp)^t\int_\bullet^1\Tilde\phi_j   
-((P^\sharp)^t+P^\sharp)\Tilde\phi=0
\end{equation}
for any $\Tilde\phi$ satisfying \eqref{e:tphi}.
Now let $u_0<u_1$ be points on $I$.
Let $U_0$ and $U_1$ be disjoint neighborhoods of $u_0$ and $u_1$ with $U_0<U_1$.
We then choose $\Tilde\phi$ to vanish outside $U_0\cup U_1$.
For $U_0<u<U_1$,
\eqref{e:voldemort} yields
\[
\dd (P^\sharp)^t(u)\sigma=0
\]
with
$\sigma:=\int_u^1\Tilde\phi$.
Since this holds for all $\sigma\in T^*_{X(0)}M$,
we see that $P$ must be constant. So now \eqref{e:voldemort}
reads
\[
((P^\sharp)^t+P^\sharp)\Tilde\phi=0.
\]
Again this must hold for all $\Tilde\phi$ satisfying \eqref{e:tphi}. {}From this we conclude 
that $P$ must be
skew-symmetric. Since these conclusions must hold for any solution $(X,\eta)$, Lemma~\ref{l:Ppi}
completes the proof.
\end{proof}
This concludes the proof of Theorem~\ref{t:one}.

\subsection{Proof of Theorem~\ref{t:two}}\label{s:proof2}
Assuming that $\pi$ is a Poisson bivector field, we know by Lemma~\ref{l:piP} that $P$ is constant and
equal to $\pi_0=\pi(X(0))$. So 
 $(\Tilde\lambda,\Tilde\phi)$ belongs to $T^\perp_{(X,\eta)}\hagrid^\text{implicit}$ if{f}
\begin{equation}\label{e:wormtail}
\int_I\braket\phi{\Tilde\lambda-\pi_0^{\sharp}\int_\bullet^1\Tilde\phi}
-\braket{\int_I\Tilde\phi}{\lambda_0}=0
\end{equation}
for all $\lambda_0\in T_0$ and $\phi\in\Oone$ such that
\begin{equation}\label{e:Ululi}
U(1)^{-1}\left(\lambda_0-\pi_0^\sharp\int_I \phi \right)\in T_1.
\end{equation}

\begin{Prop}\label{bazdeev}
Let $N_i^*=N^*_{X(i)}C_i$, $i=1,2$. If $\pi$ is Poisson, then
\begin{multline*}
T^\perp_{(X,\eta)}\hagrid^\text{\rm implicit}=
T^\perp_{(X,\eta)}\hagrid^\text{\rm explicit}:=\\
\{\Tilde\lambda\in\Ozero,\ \Tilde\phi\in \Oone : 
\dd\Tilde\lambda+\pi_0^\sharp \Tilde\phi = 0,\\ 
\Tilde\lambda(0)\in \pi_0^\sharp(N_0^*),\ 
U(1)^{-1}\Tilde\lambda(1)\in \pi_1^\sharp(N_1^*)\}.
\end{multline*}
\end{Prop}
\begin{proof}
We may first consider $\lambda_0=0$ and $\phi$ such that $\int_I\phi=0$. 
Since \eqref{e:wormtail} must hold in particular for all $(\lambda_0,\phi)$ of this kind,
we obtain that there must be a constant
$\Tilde\lambda_1\in T_{X(0)}M$ such that
\begin{equation}\label{e:patronus}
\Tilde\lambda(u)-\pi_0^\sharp
\int_u^1\Tilde\phi=\Tilde\lambda_1,
\quad \forall u\in I.
\end{equation}
So \eqref{e:wormtail} simplifies to
\begin{equation}\label{e:phil1}
\braket{\int_I\phi}{\Tilde\lambda_1}
-\braket{\int_I\Tilde\phi}{\lambda_0}=0.
\end{equation}
Observe now that $\braket{\int_I\phi}{\Tilde\lambda_1}=\braket{U(1)^t\int_I\phi}{U(1)^{-1}\Tilde\lambda_1}$.
Set $\pi_1=\pi(X(1))$. Since 
\begin{equation}\label{e:upu}
U(1)\pi_1^\sharp U(1)^t=P(1)^\sharp=\pi_0^\sharp,
\end{equation} 
we have that
$U(1)^{-1}\pi_0^\sharp\int_I\phi=\pi_1^\sharp U(1)^t\int_I\phi$. By choosing again $\lambda_0=0$, we 
get the condition
\begin{gather*}
\braket{U(1)^t\int_I\phi}{U(1)^{-1}\Tilde\lambda_1}=0,\\
\forall \phi\in\Oone \text{ such that } \pi_1^\sharp U(1)^t\int_I\phi\in T_1.
\end{gather*}
Thus,
\begin{equation}\label{e:bezuchov}
\braket\alpha{U(1)^{-1}\Tilde\lambda_1}=0,\qquad
\forall\alpha\in T^*_{X(1)}M  \text{ such that } \pi_1^\sharp \alpha\in T_1.
\end{equation}
We use now the following simple fact from linear algebra:
\begin{Lem}\label{l:F}
Let $V$ and $W$ be vector spaces. Let $F$ be a linear map $V\to W$
and $T$ a linear subspace of $W$. Then
\[
\Ann(F^{-1}(T))=F^t(\Ann(T)),
\]
where $\Ann$ denotes the annihilator of a subspace (e.g., $\Ann(T)=\{\tau\in W^*:\tau(t)=0\ \forall t\in T\}$).
\end{Lem}
\begin{proof}
It is obvious that $F^t(\Ann(T))\subset \Ann(F^{-1}(T))$. 
We now prove the other inclusion. Let $V'$ be a complement of $F^{-1}(T)$ in $V$
and $W'$ a complement of $T\oplus F(V')$ in $W$. 
Since the restriction $F|_{V'}$ of
$F$ to $V'$ establishes an isomorphism between $V'$ and $F(V')$, for any
$\psi\in\Ann(F^{-1}(T))\subset V^*$ there is a unique $\phi\in F(V')^*$ with
$\psi=F|_{V'}^t(\phi)$. Now let $\varphi\in W^*$ be equal to $\phi$ when evaluated on elements
of $F(V')$ and zero when evaluated on elements of $T$ or $W'$.
So $\varphi\in \Ann(T)$. Since $\psi=F^t(\varphi)$, this concludes the proof.
\end{proof}
We apply the Lemma to \eqref{e:bezuchov} with $V=T^*_{X(1)}M$, $W=T_{X(1)}M$, 
$T=T_1$ and $F=\pi_1^\sharp$.
Since $\pi_1$ is skew-symmetric, $F^t=-\pi_1^\sharp$.
So we get that necessarily
\[
U(1)^{-1}\Tilde\lambda_1\in\pi_1^\sharp N_1^*,
\]
where $N_1^*=N^*_{X(1)}C$ is the annihilator of $T_1$. 
So there exists $\theta\in N_1^*$ such that $\Tilde\lambda_1=U(1)\pi_1^\sharp\theta$, and
we may rewrite \eqref{e:phil1} as
\[
\braket{\int_I\phi}{U(1)\pi_1^\sharp\theta}
-\braket{\int_I\Tilde\phi}{\lambda_0}=0,
\]
or equivalently, using again \eqref{e:upu} 
and the skew-symmetry of $\pi$,
\[
\braket\theta{U(1)^{-1}\pi_0^\sharp\int_I\phi}+
\braket{\int_I\Tilde\phi}{\lambda_0}=0.
\]
This equation has to be satisfied
for all $\lambda_0\in T_0$ and $\phi\in\Oone$ satisfying \eqref{e:Ululi}.
This is equivalent to imposing
\begin{equation}\label{e:zampano}
\braket{(U(1)^t)^{-1}\theta+\int_I\Tilde\phi}{\lambda_0}=0
\end{equation}
for all $\lambda_0\in T_0$. That is,
\begin{equation}\label{e:Ulula}
(U(1)^t)^{-1}\theta+\int_I\Tilde\phi\in N_0^*,
\end{equation}
where $N_0^*=N^*_{X(0)}C_0$ is the annihilator of $T_0$. 
Recalling \eqref{e:patronus},
\begin{equation}\label{e:patronissimus}
\Tilde\lambda(u)-\pi_0^\sharp
\int_u^1\Tilde\phi
=\Tilde\lambda_1=U(1)\pi_1^\sharp\theta=\pi_0^\sharp (U^t)^{-1}\theta,
\quad \forall u\in I,
\end{equation}
we see that a pair $(\Tilde\lambda,\Tilde\phi)$ belongs to $T^\perp_{(X,\eta)}\hagrid^\text{implicit}$ 
if{f}
there exists $\theta\in N_1^*$ such that \eqref{e:Ulula} and
\begin{equation}\label{e:Uffa}
\Tilde\lambda(u)=\pi_0^\sharp\left((U^t)^{-1}\theta+\int_u^1\Tilde\phi\right),
\quad \forall u\in I,
\end{equation}
are satisfied.
By differentianting, in order to get rid of $\theta$,  we finally obtain
\[
T^\perp_{(X,\eta)}\hagrid^\text{implicit}\subset T^\perp_{(X,\eta)}\hagrid^\text{explicit}.
\]
To prove the other inclusion, consider a pair 
$(\Tilde\lambda,\Tilde\phi)\in T^\perp_{(X,\eta)}\hagrid^\text{explicit}$. Since 
$U(1)^{-1}\Tilde\lambda(1)\in \pi_1^\sharp(N_1^*)$, there exists $\theta\in N_1^*$
such that $\Tilde\lambda_1=U(1)\pi_1^\sharp\theta$. Then the solution to the equation
has the form in \eqref{e:Uffa} and satisfies \eqref{e:Ulula}.
\end{proof}

\begin{Prop}\label{twoone}
Assume $\pi$ to be Poisson.
If $C_0$ and $C_1$ are coisotropic, then so is $\hagrid$.
\end{Prop}
\begin{proof}
In this case, by \eqref{e:Ttwisted}, we immediately have
\begin{multline*}
T^\perp_{(X,\eta)}\hagrid^\text{explicit}\subset 
T_{(X,\eta)}\hagrid^\text{twisted} =\\
\{\lambda\in\Ozero,\ \phi\in \Oone : 
\dd\lambda +\pi_0^\sharp\phi=0,\\ 
\lambda(0)\in T_0,\ U(1)^{-1}\lambda(1)\in T_1\},
\end{multline*}
for all $(X,\eta)\in\hagrid$. So $\hagrid$ is coisotropic.
\end{proof}

\begin{Prop}\label{twotwo}
Assume $\pi$ to be Poisson. If $\hagrid$ is coisotropic, the so are $\Ima p_0$ and $\Ima p_1$.
\end{Prop}
\begin{proof}
If $\hagrid$ is coisotropic, then any pair 
$(\Tilde\lambda,\Tilde\phi)$ that belongs to $T^\perp_{(X,\eta)}\hagrid^\text{explicit}$ 
must also belong to $T_{(X,\eta)}\hagrid^\text{twisted}$
$\forall (X,\eta)\in\hagrid$.
Thus, in particular, we must have $\Tilde\lambda(0)\in T_0$ and $U(1)^{-1}\Tilde\lambda(1)\in T_1$.

We may arbitrarily choose the end condition $\Tilde\lambda(1)$ such that 
$U(1)^{-1}\Tilde\lambda(1)\in\pi_1^\sharp(N_1^*)$
since condition \eqref{e:Ulula}
will
always be satisfied by an appropriate choice of $\Tilde\phi$ (e.g.,
such that $\int_I\Tilde\phi=-(U(1)^t)^{-1}\theta$).
So we see that
$\pi_1^\sharp(N_1^*)\subset T_1$. Since this must hold for all $(X,\eta)\in\hagrid$,
we obtain that $N^*_x C_1\subset T_x C_1$ $\forall x\in\Ima p_1$.

Similarly, we may arbitrarily choose $\Tilde\lambda(0)\in \pi_0^\sharp(N_0^*)$
since the condition $U(1)^{-1}\Tilde\lambda(1)\in\pi_1^\sharp(N_1^*)$
will
always be satisfied by an appropriate choice of $\Tilde\phi$ (e.g.,
such that $\int_I\Tilde\phi=\tau$ if $\Tilde\lambda(0)=\pi_0^\sharp\tau$).
Then we see that
$\pi_0^\sharp(N_0^*)\subset T_0$ and,  since this must hold for all $(X,\eta)\in\hagrid$,
we obtain that $N^*_x C_0\subset T_x C_0$ $\forall x\in\Ima p_0$.
\end{proof}
This concludes the proof of Theorem~\ref{t:two}.

~

\subsubsection{Symmetries}~\label{s:sym}
From now on we assume that $\pi$ is Poisson and that $C_0$ and $C_1$ are coisotropic.
By defining
\[
b:=(U(1)^t)^{-1}\theta+\int_\bullet^1\Tilde\phi\in\Ob,
\]
we finally obtain
\begin{multline*}
T^\perp_{(X,\eta)}\hagrid
= T^\perp_{(X,\eta)}\hagrid^\text{explicit} :=\\
\{\Tilde\lambda\in\Ozero,\ \Tilde\phi\in \Oone
 :\\
\exists b\in\Ob,\ b(0)\in N_0^*,\ U(1)^tb(1)\in N_1^*,\  
\Tilde\lambda=\pi_0^\sharp b,\ \Tilde\phi=-\dd b\}.
\end{multline*}
Observe that from this description of $T^\perp_{(X,\eta)}\hagrid$ it follows immediately that
\begin{equation}\label{e:Perp}
T^{\perp\perp}_{(X,\eta)}\hagrid=T_{(X,\eta)}\hagrid.
\end{equation}

If we now invert the transformations \eqref{Ua} and \eqref{Ub} to go back to tangent vectors
at $(X,\eta)$,
\begin{align*}
\Tilde\xi &= U^{-1}\Tilde\lambda\in\Gamma(X^*TM),\\
\Tilde e &= U^t\Tilde\phi\in\Gamma(T^*I\otimes X^*T^*M),\\
\intertext{and introduce}
\beta &= U^tb\in\Gamma(X^*T^*M),
\end{align*}
we obtain that the characteristic distribution of the coisotropic submanifold $\hagrid$ is given,
at the point $(X,\eta)$, by the family of vectors $(\Tilde\xi,\Tilde e)\in T_{(X,\eta)}\hagrid$
defined by
\begin{subequations}\label{e:betasymm}
\begin{align}
\Tilde\xi &= \pi^\sharp(X)\beta,\label{e:betasymm.a}\\
\Tilde e &= -D\beta,\label{e:betasymm.b}
\end{align}
\end{subequations}
for $\beta\in \Gamma(X^*T^*M)$ with $\beta(0)\in N^*_{X(0)}C_0$ and  $\beta(1)\in N^*_{X(1)}C_1$.
In local coordinates the above formulae read,
\begin{align*}
\Tilde\xi^i &= -\pi^{ij}(X)\beta_j,\\
\Tilde e_i &= -\dd\beta_i + \Gamma_{ri}^k(X)\,\dd X^r\,\beta_k +
(\nabla \pi)^{jk}_i(X)\,\eta_j\,\beta_k.
\end{align*}
These are the symmetries of the Poisson sigma model as presented in \cite{I,SS,CF01}. The boundary
conditions for $\beta$ in case of coisotropic boundary conditions has been introduced in \cite{CF03}.

\subsubsection{``Equivariant momentum map''}
Denote by $\iota_C$ the inclusion map of a submanifold $C$ into a manifold $M$ and define
\[
\Omega_C^1(M)=\{\alpha\in\Omega^1(M) : \iota^*_C\alpha=0\}.
\]
\begin{Lem}
If $M$ is a Poisson manifold and $C$ is a coisotropic submanifold, then $\Omega_C^1(M)$ is a
Lie subalgebra of $\Omega^1(M)$.
\end{Lem}
We leave the proof of this simple fact (which directly follows from $N^*C$ being a Lie subalgebroid of $T^*M$)
to the reader.

Now, given two coisotropic submanifolds $C_0$ and $C_1$, we define the Lie algebra
\[
\Plie=\{B\colon I\to\Omega^1(M) : B(i)\in\Omega_{C_i}^1(M),\ i=0,1\},
\]
where the Lie bracket is defined pointwise. The map $B$ is assumed to be continuously differentiable.
Given a path $X$ on $M$ and an element $B$ of this Lie algebra, we define
$B_X\in\Gamma(X^*T^*M)$ by
\[
B_X(u)=B(u)(X(u)).
\]
If $X(i)\in C_i$, $i=0,1$, then $B_X(i)\in N^*_{X(i)}M$. 
It is not difficult to check that,
replacing $\beta$ by $B_X$ in \eqref{e:betasymm}, one may define an infinitesimal action of $\Plie$
on $\hagrid$ whose induced foliation is the canonical foliation. On $T^*PM(C_0,C_1)$ one may also define
Hamiltonian functions for this action; viz.,
\begin{equation}\label{e:muB}
\mu_B(X,\eta):=\int_I\braket{B_X}{\dd X+\pi^\sharp(X)\eta}.
\end{equation}

\subsection{Proof of Theorem~\ref{t:onecircle}}\label{s:proof3}
We proceed as in subsection~\ref{ss:ts}. Prop.\ \ref{p:conn} still holds, so we have
\begin{multline*}
T_{(X,\eta)}\fluffy = \{\xi\in\Gamma(X^*TM),\ e\in\Gamma(T^*S^1\otimes X^*T^*M) : \\
(D\xi)  +  \pi^\sharp(X)e =0\}.
\end{multline*}
We now regard $S^1$ as the interval $I=[0,1]$ with identified end points. The fields are then regarded as periodic sections on it.
We then continue likewise up to \eqref{e:U*} getting 
\begin{multline*}
T_{(X,\eta)}\fluffy \cong
T_{(X,\eta)}\fluffy^\text{twisted} :=\\
\{\lambda\in\Ozero,\ \phi\in \Oone : 
\dd\lambda +P^\sharp\phi=0,\\ 
\lambda(1)=U(1)\lambda(0),\ \phi(0)=U(1)^t\phi(1)\}
\end{multline*}
and
\begin{multline*}
T_{(X,\eta)}\fluffy \cong
T_{(X,\eta)}\fluffy_0 :=
\Big\{\lambda_0\in T_{X(0)}M,\ \phi\in \Oone : \\
U(1)\lambda_0 = \lambda_0 - \int_I P^\sharp\phi,\ \phi(0)=U(1)^t\phi(1)\Big\}.
\end{multline*}
We then proceed as in subsection~\ref{s:sos} getting
\begin{multline*}
T^\perp_{(X,\eta)}\fluffy \cong 
T^\perp_{(X,\eta)}\fluffy^\text{implicit} :=\\
\left\{\Tilde\lambda\in\Ozero,\ 
\Tilde\phi\in \Oone : \phantom{\int}\right.\\
\Tilde\lambda(1)=U(1)\Tilde\lambda(0),\ \Tilde\phi(0)=U(1)^t\Tilde\phi(1),\\
\int_I\braket\phi{\Tilde\lambda+(P^\sharp)^t\int_\bullet^1\Tilde\phi}
-\braket{\int_I\Tilde\phi}{\lambda_0}=0,\\
\left.\phantom{\int}
\forall(\lambda_0,\phi)\in T_{(X,\eta)}\fluffy_0\right\}.
\end{multline*}
\begin{Prop}\label{onecone}
If $\pi$ is a Poisson bivector field, then  $\fluffy$ is coisotropic.
\end{Prop}
\begin{proof}
By Lemma~\ref{l:piP} we have $P=\pi_0$, so the condition on $\lambda_0$ becomes
$U(1)\lambda_0 = \lambda_0 - \pi_0^\sharp\int_I \phi$, which is in particular satisfied by $\lambda_0=0$ assuming $\int_I\phi=0$.
In particular, we get
\[
\int_I\braket\phi{\Tilde\lambda-\pi_0^\sharp\int_\bullet^1\Tilde\phi}=0,
\]
for all $\phi\in\Oone$ satisfying $\phi(0)=U(1)^t\phi(1)$ and $\int_I\phi=0$. This implies
$\dd(\Tilde\lambda-\pi_0^\sharp\int_\bullet^1\Tilde\phi)=0$, so $(\Tilde\lambda,\Tilde\phi)\in T_{(X,\eta)}\fluffy_0$.
\end{proof}
\begin{Prop}\label{onectwo}
If $\fluffy$ is coisotropic, then  $\pi$ is a Poisson bivector field.
\end{Prop}
\begin{proof}
Notice that any $(\Tilde\lambda,\Tilde\phi)$ with $\Tilde\phi(0)=U(1)^t\Tilde\phi(1)$ satisfying \eqref{e:tlambda} and \eqref{e:tphi} belongs
to $T^\perp_{(X,\eta)}\fluffy^\text{implicit}$ and hence, since $\fluffy$ is coisotropic, to $T_{(X,\eta)}\fluffy^\text{twisted}$. The proof then proceeds exactly as in the proof to Prop.~\ref{onetwo}.
\end{proof}
This concludes the proof of Theorem~\ref{t:onecircle}.

\subsubsection{Symmetries}\label{s:symc}
Assume $\pi$ is Poisson. In the proof to Prop.~\ref{onecone} we have only considered $\lambda_0=0$ and $\int_I\phi=0$, which yields only some necessary condition to be satisfied by $(\Tilde\lambda,\Tilde\phi)$. We now want to characterize 
$T^\perp\fluffy$ completely.
\begin{Prop}
The characteristic distribution of $\fluffy$ at $(X,\eta)$ is given by the familiy of vectors
$(\Tilde\xi,\Tilde e)\in T_{(X,\eta)}\fluffy$ defined by
\begin{subequations}\label{e:betasymmcircle}
\begin{align}
\Tilde\xi &= \pi^\sharp(X)\beta,\label{e:betasymmcircle.a}\\
\Tilde e &= -D\beta,\label{e:betasymmcircle.b}
\end{align}
\end{subequations}
with $\beta\in\Gamma(X^*T^*M)$.
\end{Prop}
\begin{proof}
We have already obtained that if $(\Tilde\lambda,\Tilde\phi)$ belongs to $T^\perp\fluffy^\text{implicit}$ then, in addition to
$\Tilde\lambda(1)=U(1)\Tilde\lambda(0)$ and $\Tilde\phi(0)=U(1)^t\Tilde\phi(1)$, it satisfies 
$\dd(\Tilde\lambda-\pi_0^\sharp\int_\bullet^1\Tilde\phi)=0$. This implies that $(\Tilde\lambda,\Tilde\phi)$ belongs to $T^\perp\fluffy^\text{implicit}$ if{f} in addition 
\begin{equation}\label{e:tobe}
\braket{\int_I\phi}{\Tilde\lambda(1)}-\braket{\int_I\Tilde\phi}{\lambda_0}=0
\end{equation}
for all 
$(\lambda_0,\phi)\in T_{(X,\eta)}\fluffy_0$.

First consider $\lambda_0=0$ and $\int_I\phi\in\ker\pi_0^\sharp$. This yields, by Lemma~\ref{l:F}, that there is a $\theta\in T^*_{X(0)}M$ 
with $\Tilde\lambda(1)=\pi_0^\sharp(\theta)$; hence 
\begin{equation}\label{e:lambdatheta}
\Tilde\lambda(u)=\pi_0^\sharp\left(\theta+\int_u^1\Tilde\phi\right).
\end{equation}
As a consequence $\braket{\int_I\phi}{\Tilde\lambda(1)}=0$, so we are left with the condition
$\braket{\int_I\Tilde\phi}{\lambda_0}=0$ for all $\lambda_0$ such that there is a $\phi$ with $(\lambda_0,\phi)\in T_{(X,\eta)}\fluffy_0$.
In particular, we may take $\phi$ such that $\int_I\phi$ is in the kernel of $\pi_0^\sharp$; since 
$\pi_0^\sharp\int_I\phi=\lambda_0-U(1)\lambda_0$,
this yields that $\lambda_0$ must now lie in the kernel of the operator
$G:=U(1)-\id$. Since $\braket{\int_I\Tilde\phi}{\lambda_0}$ must vanish for all $\lambda_0$ satisying this condition, we get that
$\int_I\Tilde\phi$ must be in the image of $G^t$. Hence there is a $\gamma\in T^*_{X(0)}M$ with
\begin{equation}\label{e:phigamma}
\int_I\Tilde\phi=U(1)^t\gamma-\gamma.
\end{equation}
Using again $\pi_0^\sharp\int_I\phi=\lambda_0-U(1)\lambda_0$, we now get $\braket{\int_I\Tilde\phi}{\lambda_0}=\braket{\int_I\phi}{\pi_0^\sharp\gamma}$; so condition 
\eqref{e:tobe}
finally reads 
$\braket{\int_I\phi}{\Tilde\lambda(1)-\pi_0^\sharp\gamma}=0$
for all $\phi$ such that there is a $\lambda_0$ with $(\lambda_0,\phi)\in T_{(X,\eta)}\fluffy_0$. Using \eqref{e:lambdatheta},
the condition becomes $\braket{\theta-\gamma}{\pi_0^\sharp\int_I\phi}=0$ or, using $\pi_0^\sharp\int_I\phi=-G\lambda_0$,
$\braket{\theta-\gamma}{G\lambda_0}=0$. Hence we have that $\theta-\gamma$ must be in the annihilator of $\im\pi_0^\sharp\cap\im G$
which is $\ker\pi_0^\sharp+\ker G^t$. We hence have $\mu,\nu\in T_{X(0)}^*M$ with  $\pi_0^\sharp\mu=0$ and $U(1)^t\nu=\nu$ such that
$\theta-\gamma=\mu+\nu$.
Finally, define 
\[
b(u):=\theta-\mu + \int_u^1\Tilde\phi = \nu + \gamma + \int_u^1\Tilde\phi.
\]
Since $U(1)^t\nu=\nu$, we have thanks to \eqref{e:phigamma} that $b(0)=U(1)^tb(1)$. This shows that
\[
\beta:=U^tb
\]
is a periodic section of $X^*T^*M$.
Since $\pi_0^\sharp\mu=0$, we have from \eqref{e:lambdatheta} that $\Tilde\lambda=\pi_0^\sharp b$. 
Moreover, since $\theta$ and $\mu$ are constant,
we have $\Tilde\phi=-\dd b$.

If we now invert the transformations \eqref{Ua} and \eqref{Ub} to go back to tangent vectors
at $(X,\eta)$,
\begin{align*}
\Tilde\xi &= U^{-1}\Tilde\lambda\in\Gamma(X^*TM),\\
\Tilde e &= U^t\Tilde\phi\in\Gamma(T^*S^1\otimes X^*T^*M),
\end{align*}
we obtain that a pair $(\Tilde\xi,\Tilde e)$ belongs to $T^\perp\fluffy$ if{f} equations \eqref{e:betasymmcircle.a} and \eqref{e:betasymmcircle.b}
are satisfied.
\end{proof}
\begin{Rem}
One may define functions $\mu_B$ on $T^*LM$ as in \eqref{e:muB} with $B$ now a map from $S^1$ to $\Omega^1(M)$. Notice that the functions $\mu_B$ generate the vanishing ideal of $\fluffy$ and that their Hamiltonian vector fields generate the distribution defined by \eqref{e:betasymmcircle.a} and \eqref{e:betasymmcircle.b}. This remark however does not replace the proof above as in the infinite dimensional case it is not automatic that the Hamiltonian vector fields of functions in the vanishing ideal span the whole characteristic distribution.
\end{Rem}

\section{Dual pairs}\label{s:dual}
In this Section we assume that $C_0$ and $C_1$ are coisotropic submanifolds of a Poisson manifold $M$.
In this case $\hagrid$ is a coisotropic submanifold of $T^*PM$ and
its leaf space $\underline\hagrid$ is endowed with a symplectic structure. On the other hand,
the leaf spaces $\underline{C_0}$ and $\underline{C_1}$ are endowed with a Poisson structure.
By \eqref{e:betasymm.a} and the conditions on $\beta$,
the maps $p_i$ of \eqref{e:p} descend to the quotients
\[
\up_i\colon\underline{\hagrid}\to \underline{C_i},
\qquad i=0,1.
\]
Proceeding as in the proof of Theorem~4.6 in \cite{CF01}, one may prove
that $\up_0$ and $\up_1$ are a Poisson and an anti-Poisson map respectively.
We will prove the following
\begin{Lem}\label{l:dementor}
$\ker\dd\up_0$ and $\ker\dd\up_1$ are symplectically orthogonal at any point of $\underline\hagrid$,
i.e.,
\[
(\ker\dd\up_0(x))^\perp = \ker\dd\up_1(x),
\qquad \forall x\in\underline\hagrid.
\]
\end{Lem}
In other words, $\underline{C_0}\stackrel{{\up_0}}{\longleftarrow}\underline\hagrid\stackrel{\up_1}{\longrightarrow}\underline{C_1}$
is is a Lie--Weinstein dual pair \cite{W}. Observe that 
the maps $\up_i$ may fail to be
surjective submersions, so this dual pair is in general not full.
\begin{Rem}
Since the quotient  $\underline\hagrid$ is finite dimensional, the above condition is equivalent to
\[
(\ker\dd\up_1(x))^\perp = \ker\dd\up_0(x),
\qquad \forall x\in\underline\hagrid.
\]
\end{Rem}
Notice that the maps $\up_i$s are defined, as continuous maps, 
even if the leaf spaces are not smooth. Lemma~\ref{l:dementor} 
makes sense also in the nonsmooth case if we define $T_{[x]}\underline{\hagrid}$
as $\underline{T_x\hagrid}$, $x\in[x]\in\underline{\hagrid}$, and
$T_{[x]}\underline{C_i}$ as $\underline{T_x C_i}$, $x\in[x]\in\underline{C_i}$.
The linear maps $\ker\dd\up_0$ and $\ker\dd\up_1$ are also well defined.
Thus, we may think of $\underline{C_0}\stackrel{{\up_0}}{\longleftarrow}\underline\hagrid\stackrel{\up_1}{\longrightarrow}\underline{C_1}$
as of a singular Lie--Weinstein dual pair. 
%
%
\begin{proof}[Proof of Lemma~\ref{l:dementor}]
Let $(X,\eta)$ be a representative of $x\in\underline\hagrid$.
We introduce the following notations:
\begin{align*}
V &:= T_{(X,\eta)}\hagrid,\\
Z_i &:= T_{X(i)}C_i,\\
Z_i^\perp &:= \pi^\sharp(X(i))(N^*_{X(i)}C_i),\\
\varpi_i &:= \dd p_i(X,\eta),\\
\end{align*}
for $i=0,1$. So we have the following commutative diagram of vector spaces:
\[
\begin{CD}
V @>{\varpi_i}>> Z_i \\
@VVV @VVV\\
V/V^\perp @>>{\dd\up_i(x)}> Z_i/Z_i^\perp
\end{CD}
\]
We then have
\begin{align*}
\ker\dd\up_i(x)&=\varpi_i^{-1}(Z_i^\perp)/V^\perp,\\
(\ker\dd\up_i(x))^\perp &=(\varpi_i^{-1}(Z_i^\perp))^\perp/V^\perp,
\end{align*}
and
\begin{align*}
\varpi_0^{-1}(Z_0^\perp)=
\{\lambda\in\Ozero,\ \phi\in \Oone : 
\dd\lambda +\pi_0^\sharp\phi=0,\\ 
\lambda(0)\in \pi_0^\sharp(N_0^*),\ U(1)^{-1}\lambda(1)\in T_1\},\\
\varpi_1^{-1}(Z_1^\perp)=
\{\lambda\in\Ozero,\ \phi\in \Oone : 
\dd\lambda +\pi_0^\sharp\phi=0,\\ 
\lambda(0)\in T_0,\ U(1)^{-1}\lambda(1)\in \pi_1^\sharp(N_1^*)\}.
\end{align*}

\subsubsection*{Step 1:} $(\ker\dd\up_1(x)) \subset \ker\dd\up_0(x)^\perp$.

It is enough to show that
\[
\Omega_{(X,\eta)}((\xi,\zeta),(\Tilde\xi,\Tilde \zeta))=0,\quad
\forall (\xi,\zeta)\in \varpi_0^{-1}(Z_0^\perp),\ 
(\Tilde\xi,\Tilde\zeta)\in \varpi_1^{-1}(Z_1^\perp),
\]
since this implies
$\varpi_1^{-1}(Z_1^\perp)^\perp\subset\varpi_0^{-1}(Z_0^\perp)^\perp$
which in turn implies the desired result on the quotient.
Using \eqref{e:ana} with
\begin{align*}
\lambda &= \lambda_0 -\pi_0^\sharp\int_0^\bullet\phi,\\
\Tilde\lambda &= \Tilde\lambda_1 +\pi_0^\sharp\int_\bullet^1\Tilde\phi,
\end{align*}
we get
\[
\Omega_{(X,\eta)}((\xi,\zeta),(\Tilde\xi,\Tilde \zeta))=
\braket{\int_I\phi}{\Tilde\lambda_1} -\braket{\int_I\Tilde \phi}{\lambda_0}.
\]
If we now write
\begin{align*}
\lambda_0 &= \pi_0^\sharp\alpha,\\
\Tilde\lambda_1 &= U(1)\pi_1^\sharp\beta = \pi_0^\sharp (U^t)^{-1}\beta,
\end{align*}
with $\alpha\in N_0^*$ and $\beta\in N_1^*$, we get
\[
\braket{\int_I\phi}{\Tilde\lambda_1} = \braket{\int_I\Tilde \phi}{\lambda_0} =
-\braket{U(1)^{-1}\pi_0^\sharp\alpha}\beta,
\]
which completes Step~1.

\subsubsection*{Step 2:} $(\ker\dd\up_0(x))^\perp \subset \ker\dd\up_1(x)$.

Let $(\Tilde\lambda,\Tilde\phi)$ be an element of $(\varpi_0^{-1}(Z_0^\perp))^\perp$.
Then proceeding exactly as in the proof of Proposition~\ref{bazdeev}, we see that
this element must satisfy \eqref{e:patronus} with 
$\Tilde\lambda_1=U(1)\pi_1^\sharp\theta$ for some $\theta\in N_1^*$.
We still have condition \eqref{e:zampano} but now for all $\lambda_0\in\pi_0^\sharp(N_0^*)$.
This implies that
\[
(U^t)^{-1}\theta+\int_I\Tilde\phi\in\Ann(\pi_0^\sharp N_0^*).
\]
In finite dimensions, Lemma~\fullref{l:F} implies $F^{-1}(T)=\Ann(F^t(\Ann(T)))$.
Taking $T=T_0$ and $F=\pi_0^\sharp$, we then get
\[
\Ann(\pi_0^\sharp N_0^*) = \Ann(\pi_0^\sharp\Ann(T_0))=
(\pi_0^\sharp)^{-1}(T_0).
\]
Thus, \eqref{e:Uffa} implies
\[
\Tilde\lambda(0)=\pi_0^\sharp\left((U^t)^{-1}\theta+\int_I\Tilde\phi\right)\in T_0.
\]
Hence $(\Tilde\lambda,\Tilde\phi)$ is an element of $\varpi_1^{-1}(Z_1^\perp)$,
and its class modulo $V^\perp$ is an element of $\ker\dd\up_1(x)$.
\end{proof}

\subsection{Composition}
Under certain technical conditions (see \cite{L} and references therein), dual pairs can be composed by symplectic reduction. Namely, let $P_0$, $P_1$ and $P_2$ be Poisson
manifolds, $S_0$ and $S_1$ symplectic manifolds, together with Poisson maps
$I_0$, $J_0$, and anti-Poisson maps $I_1$ and $J_1$ as in the following diagram:
\[
\begin{CD}
P_0 @<{I_0}<< S_0 @>{I_1}>> P_1 @<{J_0}<< S_1 @>{J_1}>> P_2,
\end{CD}
\]
then $S_0\times_{P_1} S_1$ is a coisotropic submanifold of $S_0\times \overline{S_1}$
($\overline{S_1}$ denotes $S_1$ with opposite symplectic structure), and
the maps $I_0$, $J_1$ descend to the symplectic quotients, so that
\[
\begin{CD}
P_0 @<\underline{I_0}<< \underline{S_0\times_{P_1} S_1} @>\underline{J_1}>> P_2
\end{CD}
\]
is a new dual pair which we will denote by
\[
S_0\star S_1.
\]
Of course, without the appropriate assumptions,
this might be quite singular; even if we started with smooth manifolds,
already the fibered product $S_0\times_{P_1} S_1$ might not be a manifold,
unless $I_1$ and $J_0$ are surjective submersions. For the reduced space to be smooth
as well, one need some more assumptions, see \cite{L}.

In our case, we allow all sorts of singularity. Given coisotropic submanifolds $C_0$,
$C_1$ and $C_2$, we can construct singular dual pairs by $\underline\hagrid$
and $\underline\hagridbis$. A natural question is whether
$\underline{\hagr{C_0}{C_2}}=\underline\hagrid\star\underline\hagridbis$.
Roughly speaking the composition of these dual pairs arises by joining paths
at a fiber of $C_1\to \underline{C_1}$. So we may expect the above identity to hold
only if every path from $C_0$ to $C_2$ is equivalent to a path that passes through
$C_1$. Otherwise in the composition we will select only paths with this property, so we may
expect that, in general, only the following inclusion relation holds:
\[
\underline{\hagr{C_0}{C_2}}\supset\underline\hagrid\star\underline\hagridbis.
\]
In the next Section, through another description of our singular dual pairs,
this will be more clear. Another way out is the extension to this case of the construction in \cite{CC,Contrthesis}, 
where we might speak of relational dual pairs. 

\section{Reduced spaces}\label{s:fd}
Let $(M,\pi)$ be a Poisson manifold, then $\G(M)=\underline{\hagr MM}$ is the
(possibly singular) source-simply-connected symplectic groupoid of $M$ \cite{CF01}.
In this case we will denote by $s$ and $t$ (instead of $\up_0$ and $\up_1$)
the Poisson and anti-Poisson maps to $M$. Given two submanifolds $C_0$ and $C_1$
of $M$, we define
\[
\rubeus = s^{-1}(C_0)\cap t^{-1}(C_1).
\]
If $C_0$ and $C_1$ are coisotropic, then so are $s^{-1}(C_0)$, $t^{-1}(C_1)$
and (because of the symplectic orthogonality of the $s$\ndash\ and $t$\ndash fibers)
also $\rubeus$. We may then consider its reduction $\underline\rubeus$.
We have the following
\begin{Thm}\label{t:rubeushagrid}
$\underline\rubeus=\underline\hagrid$.
\end{Thm}
This Theorem is a consequence of the following
\begin{Lem}[Reduction in stages]
Let $S$ be a (possibly infinite-\hspace{0pt}dimensional, weak) symplectic space.
Let $V$ be a subspace of $S$ and $W$ a subspace of $V$.
If $W$ is coisotropic in $S$, then:
\begin{enumerate}
\item $V^\perp\subset W^\perp\subset W\subset V$, and in particular
$V$ is also coisotropic.
\item $W/V^\perp$ is coisotropic in $V/V^\perp$.
\item $W/W^\perp=(W/V^\perp)/(W/V^\perp)^\perp$.
\end{enumerate}
\end{Lem}
The proof is a simple exercise in linear algebra.

\begin{proof}[Proof of Theorem~\ref{t:rubeushagrid}]
Let $x$ be a point in $\rubeus\subset\G(M)$ and $(X,\eta)$ 
a representative of $x$ in $\hagrid\subset\hagr MM$.
Then we apply the reduction in stages to the following spaces
\begin{align*}
S &:= T_{(X,\eta)}T^*PM,\\
V &:=  T_{(X,\eta)}\hagr MM,\\
W &:= T_{(X,\eta)}\hagrid,\\
\intertext{observing that}
W/V^\perp &= T_x\rubeus.
\end{align*}
\end{proof}

Thanks to Theorem~\ref{t:rubeushagrid} we may now easily discuss a few examples.

\begin{Exa}[Trivial Poisson structure]
For $\pi=0$ we have $\G(M)=T^*M$ with canonical symplectic structure and with
$s=t=$projection $T^*M\to M$.
Any submanifold of $M$ is automatically coisotropic with trivial foliation.
Then we have
\begin{align*}
\calC_0(M;C_0,C_1) &= T^*_{C_0\cap C_1} M\\
\intertext{and}
\underline{\calC_0(M;C_0,C_1)} &= T^*(C_0\cap C_1)
\end{align*}
which is a manifold if{f} $C_0\cap C_1$ is so. Moreover, a simple computation shows that
\[
\underline{\calC_0(M;C_0,C_1)}\star\underline{\calC_0(M;C_1,C_2)}=
T^*(C_0\cap C_1\cap C_2)\subset \underline{\calC_0(M;C_0,C_2)}.
\]
\end{Exa}
\begin{Exa}[Symplectic case]
Let $M$ be a symplectic manifold, and $\pi$ the corresponding Poisson structure.
For simplicity we assume $M$ to be simply connected. Then
$\G(M)=M\times\overline M$, where $\overline M$ denotes $M$ with opposite symplectic structure.
The maps $s$ and $t$ are the projections to the factors.
Thus,
\begin{align*}
\hagrid &= C_0\times C_1,\\
\intertext{and, in case $C_0$ and $C_1$ are coisotropic,}
\underline\hagrid &= \underline{C_0}\times \underline{C_1}.
\end{align*}
In this case,
\[
\underline{\hagr{C_0}{C_2}}=\underline\hagrid\star\underline\hagridbis,
\]
for any three coisotropic submanifolds $C_0$, $C_1$ and $C_2$.
\end{Exa}


Observe that the map to $M$ from $\rub{C_0}M$ (resp., $\rub M{C_1}$) is a surjective
submersion if $C_0$ (resp., $C_1$) has a clean intersection with every symplectic leaf of $M$.
Under this condition, the dual pairs behave well and the composition is well defined
(in the world of differentiable stacks). Thus, if we define a coisotropic
submanifold to be nice when it has a clean intersection with every symplectic leaf,
we have the following
\begin{Thm}
To every Poisson manifold $(M,\pi)$, we may associate a category $\C(M)$ where the objects
are Poisson reductions (as differentiable stacks)
of nice coisotropic submanifolds of $M$
and the morphisms from the object $\underline{C_0}$ to the
object $\underline{C_1}$ are 
elements of the form $\underline{\rubeus}$ (as a differentiable stack)
where $C_0$ and $C_1$ are nice coisotropic submanifolds of $M$ which
reduce to $\underline{C_0}$ and $\underline{C_1}$.
\end{Thm}

\subsection{Groupoid quotients}
As shown in \cite{C}, there is a one-to-one correspondence between
coisotropic submanifolds of a given Poisson manifold $M$ and
(possibly singular) Lagrangian subgroupoids of the (possibly singular)
source-simply-connected symplectic groupoid $\G(M)$ of $M$.
It turns out that $\underline\rubeus$ may be understood as a quotient of groupoids.
Namely, given a groupoid $\G\toto M$ and a subgroupoid $\LLL\toto C\subset M$,
we have a left action of $\LLL$ on $s^{-1}(C)$ and a right action on $t^{-1}(C)$.
We will write $\LLL\backslash\G$ (resp., $\G/\LLL$) as a shorthand notation for
$\LLL\backslash s^{-1}(C)$ (resp., $t^{-1}(C)/\LLL$). If $\LLL(C)$ is the Lagrangian
subgroupoid of $\G(M)$ corresponding to the coisotropic submanifold $C$, we have
\begin{align*}
\underline{\hagr CM} &= \LLL(C)\backslash\G(M),\\
\underline{\hagr MC} &= \G(M)/\LLL(C).\\
\intertext{For two given coisotropic submanifolds $C_0$ and $C_1$, we have instead}
\underline\hagrid &=  \LLL(C_0)\backslash\G(M)/\LLL(C_1):=\\
&=\LLL(C_0)\backslash(s^{-1}(C_0)\cap t^{-1}(C_1))/\LLL(C_1).
\end{align*}
This can be verified by recalling the construction of \cite{C} and comparing it with the one
in this paper. Namely, the results of subsection~\ref{s:sym}, may be rephrased as follows:
\begin{multline*}
\hagrid=\{\text{Lie algebroid morphisms $TI\to T^*M$}\\
\text{with base maps connecting
$C_0$ to $C_1$}\}.
\end{multline*}
Two elements $\gamma_0$ and $\gamma_1$ of $\hagrid$ are defined to be equivalent
if there exists a Lie algebroid morphism
$\Gamma\colon T(I\times J)\to T^*M$, with $J$ an interval, such that
\begin{enumerate}
\item the restriction of $\Gamma$ to the boundaries of $J$ are the two given
morphisms $\gamma_0$ and $\gamma_1$, and
\item the restriction of $\Gamma$ to the boundaries $\{0\}$ and $\{1\}$ of $I$ are Lie
algebroid morphism $TJ\to N^*C_i$, $i=0,1$.
\end{enumerate}
Then $\underline\hagrid$ may be regarded as the quotient of $\hagrid$ by this
equivalence relation.

On the other hand, the Lagrangian subgroupoid corresponding to a coisotropic submanifold
$C$ is shown in \cite{C} to be the source-simply-connected groupoid whose Lie algebroid
is $N^*C$, and this is exactly \cite{CF}
the quotient of the space of Lie algebroid morphisms
$TJ\to N^*C$ by Lie algebroid morphisms $T(J\times K)\to N^*C$ (where $K$ is another interval)
which are trivial on the boundary of $J$. Finally observe that equivalent  Lie algebroid 
morphisms $TJ\to N^*C_i$ act the same way on $\hagrid$.

\begin{Rem}
This construction may be generalized to any Lie groupoid.
\end{Rem}

\section{Lagrangian field theories with boundary}\label{s:LFT}
As mentioned in subsection~\ref{ss:LFT}, the constructions in this paper are related to the general ones for Lagrangian field theories on manifolds with boundary as
in \cite{CMR} and \cite{CMR2}.

We start recalling a few facts.
In general to a compact oriented manifold $N$ (possibly with boundary), of fixed dimension $d$, the theory associates a space of fields $\calF_N$ and a function $S_N$ on $\calF_N$ called the action functional.
The case of this paper, fixing a manifold $M$ and a tensor $\pi$ on it, 
corresponds to $d=2$, $\calF_N$ the space of bundle maps $TN\to T^*M$ and
\[
S_N(X,\eta)=\int_N \braket\eta{\dd X} + \frac12 \braket\eta{\pi^\sharp\eta}
\]
where $X$ is a map $N\to M$, $\eta$ a section of $T^*N\otimes X^*T^*M$ and again $\braket{\ }{\ }$ is
the canonical pairing between the cotangent and the tangent bundles to $M$. Notice that the action functional does not see the symmetric part of $\pi$, so it may be convenient to assume that $\pi$ is a bivector field. When $\pi$ is Poisson, this theory is called the Poisson sigma model \cite{I,SS}. We will call
the general case the bivector sigma model (BSM).

Under certain assumptions (in particular locality), to a compact oriented $(d-1)$\ndash manifold $\Sigma$
the theory also associates an exact weak symplectic manifold  $(\calF^\de_\Sigma,\Omega_\Sigma=\dd\Theta_\Sigma)$ such that
whenever $\Sigma=\de M$ for a compact $d$\ndash manifold $M$ there is a surjective submersion $\pi_M\colon \calF_M\to\calF^\de_{\de M}$.
Denoting by $EL_M$ the zero set\footnote{This is the set of solutions to Euler--Lagrange equations where one ``ignores" the boundary.} 
of the $1$\ndash form $\EL_M:=\dd S_M-\pi_M^*\Theta_{\de M}$, one obtains that the ``evolution relation"
$L_M:=\pi_M(EL_M)$ is isotropic in $\calF^\de_{\de M}$. 
In ``good" theories the $L_M$s should be Lagrangian. Observe that $L_M$ is here defined just as a subset; a good additional condition is that it should be a (possibly immersed) submanifold.
One also has $\calF^\de_{\Sigma\sqcup\Sigma'}=\calF^\de_{\Sigma}\times \calF^\de_{\Sigma'}$ as a product of exact weak symplectic manifolds and
$\calF^\de_{\Sigma^\text{op}}=\overline{\calF^\de_{\Sigma}}$, where $\Sigma^\text{op}$ denotes $\Sigma$ with opposite orientation and bar
denotes the same manifold with opposite one\ndash form. Finally one defines the ``space of Cauchy data"
$\calC_\Sigma$ as the space of points in $\calF^\de_{\Sigma}$ that
can be completed to a pair of points in $L_{\Sigma\times[0,\epsilon]}\subset\calF^\de_{\Sigma}\times\overline{\calF^\de_{\Sigma}}$ 
for some $\epsilon>0$. Under some mild assumptions,
one can show that, if $L_{\Sigma\times[0,\epsilon]}$
is Lagrangian for all $\epsilon$, then $\calC_\Sigma$ is coisotropic.\footnote{The first assumption is that the theory behaves well under diffeomorphims: viz.,
a diffeomorphism of the bulk manifolds induces a diffeomorphism of the corresponding $EL$ spaces. This implies that each $L_{\Sigma\times[0,\epsilon]}$ is
symmetric: $(x,y)\in L_{\Sigma\times[0,\epsilon]}$ if{f} $(y,x)\in L_{\Sigma\times[0,\epsilon]}$. 
The second assumption is locality which in particular implies that
we can restrict solutions; so, if we know that $x$ lies in $\calC_\Sigma$ because there is a $y$ with $(x,y)\in L_{\Sigma\times[0,\epsilon]}$, 
then we also know that for all
$\epsilon'<\epsilon$ there is a $y'$ with $(x,y')\in L_{\Sigma\times[0,\epsilon']}$. (In the case of topological field theories, like the BSM, this part of the argument
is much easier since $L_{\Sigma\times[0,\epsilon]}=L_{\Sigma\times[0,\epsilon']}$ for all $\epsilon,\epsilon'$.)

Suppose now that $x$ lies in $\calC_\Sigma$. Pick a compact neighborhood $U_x$ of $x$ in $\calC_\Sigma$. For each $z$ in $U_x$ there is then
an $\epsilon_z>0$ and a $y$ such that $(z,y)\in L_{\Sigma\times[0,\epsilon_z]}$. Let $\epsilon$ be the maximum $\epsilon_z$ for $z$ in $U_x$.
Thanks to the locality assumption,
we then have the simplified statement that for all $z\in U_x$ there is a $y_z$ such $(z,y_z)\in L_{\Sigma\times[0,\epsilon]}$.
This in particular shows that $T_x\calC_\Sigma$ consists of all $v\in T_x\calF^\de_{\Sigma}$ such that there is a $w\in T_y\calF^\de_{\Sigma}$, with $y:=y_x$,
such that $(v,w)\in T_{(x,y)}L_{\Sigma\times[0,\epsilon]}$. Thanks to the symmetry property, we also see that 
$T_{(x,y)}L_{\Sigma\times[0,\epsilon]}$ is contained in $T_x\calC_\Sigma\oplus T_y\calC_\Sigma$. The orthogonal space of the latter, 
in $T_x\calF^\de_{\Sigma}\oplus \overline{T_y\calF^\de_{\Sigma}}$, is readily seen to be $(T_x\calC_\Sigma)^\perp\oplus (T_y\calC_\Sigma)^\perp$.
Since $L_{\Sigma\times[0,\epsilon]}$ is Lagrangian by assumption, we conclude that $(T_x\calC_\Sigma)^\perp\oplus (T_y\calC_\Sigma)^\perp$
is contained in $T_x\calC_\Sigma\oplus T_y\calC_\Sigma$, so $T_x\calC_\Sigma$ is coisotropic (and so is $T_y\calC_\Sigma$). Since this can be shown for all
$x\in\calC_\Sigma$, we have that $\calC_\Sigma$ is coisotropic.}
In the BSM, we get $(\calF^\de_{S^1},\Theta_{S^1})=(T^*LM,\Theta)$ and 
$\calC_{S^1}=\fluffy$ with the notations of subsection~\ref{ss:caci}.
We have the following
\begin{Thm}
$L_{S^1\times[0,\epsilon]}$ is an immersed Lagrangian submanifold $\forall\epsilon>0$ if{f}\, $\pi$ is Poisson.
\end{Thm}
\begin{proof}
If $L_{S^1\times[0,\epsilon]}$ is Lagrangian $\forall\epsilon>0$, then $\calC_{S^1}=\fluffy$ is coisotropic by the general theory. By Theorem~\ref{t:onecircle}, $\pi$ is then Poisson.

On the other hand, if $\pi$ is Poisson, then $\calC_{S^1}=\fluffy$ is coisotropic by Theorem~\ref{t:onecircle}. 
In addition, $\forall\epsilon>0$, $L_{S^1\times[0,\epsilon]}$ 
consists of all pairs $(x,x')\in\fluffy\times\fluffy$ such that $x$ and $x'$ are on the same characteristic leaf.\footnote{The critical points of $S_N$
are solutions to
\begin{align*}
\dd X &= -\pi^\sharp(X)\eta,\\
\dd\eta_i &= \frac12\de_i\pi^{jk}(X)\eta_j\eta_k.
\end{align*}
For simplicity of notations we work in local coordinates (the rest of the computation may be done in covariant form since, using the first equation, the second
can be written as $\partial\eta_i = \frac12(\nabla\pi)_i^{jk}(X)\eta_j\eta_k$).
We write $\eta=\eta_\parallel-\beta\dd t$, where $\eta_\parallel$ is a $1$\ndash form in the $S^1$ direction, $\beta$ a function and $t$ the coordinate on $[0,\epsilon]$.
Writing $\dd=\dd_\parallel+\dd t\delta$, with $d_\parallel$  the differential in the $S^1$ direction and $\delta$ the partial derivative with respect to $t$,
we may rewrite the equations
as
\begin{align*}
\dd_\parallel X &= -\pi^\sharp(X)\eta_\parallel,\\
\delta X &= \pi^\sharp(X)\beta,\\
\delta\eta_{\parallel i} &= -\dd_\parallel\beta_i
+\de_i\pi^{jk}(X)\beta_j\eta_{\parallel k}.
\end{align*}
The first equation says that the restriction of $(X,\eta)$ to $S^1\times\{t\}$ for each $t$ yields an element of $\calC_\pi(M)$. The two other equations
say that $(X,\eta)$ evolves in the $t$ direction along the characteristic distribution of $\calC_\pi(M)$, cf.\ \eqref{e:betasymmcircle}.} 
This implies that $L_{S^1\times[0,\epsilon]}$ is an immersed Lagrangian
submanifold.\footnote{We follow the analogue proof in \cite[Section 3.5.2]{Contrthesis}).
Since $L_{S^1\times[0,\epsilon]}$ consists of pair of points on the same leaf of the characteristic distribution of $\calC_\pi(M)$, it follows that it is an immersed submanifold. We now prove that it is Lagrangian.
Let $(x,y)$ be in $L_{S^1\times[0,\epsilon]}$. Then, in addition to kowing that
$T_{(x,y)}L_{S^1\times[0,\epsilon]}$ is a subspace of $T_x\calC_\pi(M)\oplus T_y\calC_\pi(M)$,
which in turn is a coisotropic subspace of $T_x(T^*LM)\oplus\overline{T_y(T^*LM)}$, we now also know that it contains 
$(T_x\calC_\pi(M))^\perp\oplus (T_y\calC_\pi(M))^\perp=(T_x\calC_\pi(M)\oplus T_y\calC_\pi(M))^\perp$. Moreover, 
$\underline{T_{[x]}\calC_\pi(M)}:=T_x\calC_\pi(M)/(T_x\calC_\pi(M))^\perp$ gets canonically identified with $T_y\calC_\pi(M)/(T_y\calC_\pi(M))^\perp$. Finally,
$T_{(x,y)}L_{S^1\times[0,\epsilon]}/(T_x\calC_\pi(M)\oplus T_y\calC_\pi(M))^\perp$ is the diagonal in 
$\underline{T_{[x]}\calC_\pi(M)}\oplus\overline{\underline{T_{[x]}\calC_\pi(M)}}$, which is Lagrangian. This proves that
$T_{(x,y)}L_{S^1\times[0,\epsilon]}$ itself is Lagrangian, see \cite[Proposition A.1(3)]{CMR}.
}
\end{proof}

This story extends to the case when $\Sigma$ is only part of $\de M$. If $\de M=\Sigma\sqcup\Sigma'$, then $\pi_M$ is the product of
two surjective submersions $\pi_{M,\Sigma}$ and
$\pi_{M,\Sigma'}$ to $\calF^\de_{\Sigma}$ and $\calF^\de_{\Sigma'}$, respectively.
Upon picking a Lagrangian submanifold $L'$ of $\calF^\de_{\Sigma'}$ on which $\Theta_{\Sigma'}$ vanishes, one sets
$\calF_M^{L'}:=\pi_{M,\Sigma'}^{-1}(L')$. Denoting by $S_M^{L'}$, $\EL_M^{L'}$ and $\pi_{M,\Sigma}^{L'}$ 
the restrictions of $S_M$, $\EL_M$ and $\pi_{M,\Sigma}$ to $\calF_M^{L'}$, one then has
$\EL_M^{L'}=\dd S_M^{L'}-(\pi_{M,\Sigma}^{L'})^*\Theta_{\Sigma}$ and therefore $L_M^{L'}:=\pi_{M,\Sigma}^{L'}(EL_M^{L'})$ is isotropic,
where $EL_M^{L'}$ denotes the zero set of $\EL_M^{L'}$.

A further extension occurs when $M$ is a compact manifold with corners and the codimension\ndash one boundary stratum of $M$, which we denote by $\de M$,
is the union of compact manifolds with boundary $\Sigma$ and $\Sigma'$ joined along their common boundary, the codimension\ndash two boundary stratum of $M$
(we assume that there are no further lower dimensional boundary strata).  The story described in the previous paragraph extends verbatim to this case.

Finally, one can define the ``space of Cauchy data" $\calC_\Sigma^{L'}$ as the space of points in $\calF^\de_{\Sigma}$ that
can be completed to a pair of points in $L_{\Sigma\times[0,\epsilon]}^{L'}\subset\calF^\de_{\Sigma}\times\overline{\calF^\de_{\Sigma}}$ 
for some $\epsilon>0$, where $L'$ is a fixed Lagrangian submanifold of $\calF_{\de\Sigma\times[0,\epsilon]}$.\footnote{Since a diffeomorphism
$[0,\epsilon]\to[0,\epsilon']$ induces a symplectomorphism
$\calF^\de_{\gamma\times[0,\epsilon]}\to\calF^\de_{\gamma\times[0,\epsilon']}$, it is enough to select $L'$ in $\calF_{\gamma\times[0,\epsilon]}$ for a fixed $\epsilon$.} Again, one shows that, if $L_{\Sigma\times[0,\epsilon]}^{L'}$
is Lagrangian for all $\epsilon$, then $\calC_\Sigma^{L'}$ is coisotropic. 

In the BSM, we get $(\calF^\de_{I},\Theta_I)=(T^*PM,\Theta)$ with the notations of the Introduction.
Moreover, for a submanifold $C$ of $M$, the space $L_C$ of bundle maps $T[0,\epsilon]\to N^*C$
is a Lagrangian submanifold of $\calF^\de_{[0,\epsilon]}$. Finally, one has $\calC_I^{L_{C_0}\times L_{C_1}}=\hagrid$ and
\begin{Thm}
$L_{I\times[0,\epsilon]}^{L_{M}\times L_{M}}$ is an immersed Lagrangian submanifold $\forall\epsilon>0$ if{f}\, $\pi$ is Poisson.
If $\pi$ is Poisson and $C_0$ and $C_1$ are coisotropic, then $L_{I\times[0,\epsilon]}^{L_{C_0}\times L_{C_1}}$ is Lagrangian $\forall\epsilon>0$.
\end{Thm}
The proof is similar to the case of $S^1$ but now uses Theorem~\ref{t:one}, Corollary~\ref{cor:cor} and the results of subsection~\ref{s:sym}.



%
%
%
%
%

\end{document}